\documentclass[11pt,a4paper,english]{article}
\usepackage{}
\usepackage{amsfonts}
\usepackage{amssymb}
\usepackage{mathrsfs}
\usepackage{amsmath}
\usepackage{booktabs}
\usepackage{epsf,epsfig,amsfonts,amsgen,indentfirst}
\usepackage{amsmath,amstext,amsbsy,amsopn,amsthm,bbding,wasysym}
\usepackage{multicol,mathdots}

\newcommand{\F}{\mbox{$\cal F$}}

\newtheorem{theorem}{Theorem}[section]

\newtheorem{lemma}[theorem]{Lemma}
\newtheorem{Definition}[theorem]{Definition}
\newtheorem{cor}[theorem]{Corollary}

\begin{document}
\title
{\LARGE \textbf{On the permanental nullity and  matching number of graphs\thanks{This work is supported by NSFC (11371180) and a project of QHMU(2015XZJ12).} }}

\author{ Tingzeng Wu$^{a}$\thanks{
\emph{E-mail address}: mathtzwu@163.com, hjlai@math.wvu.edu}, Hong-Jian Lai$^{b}$\\
{\small $^{a}$School of Mathematics and Statistics, Qinghai Nationalities University, }\\
{\small  Xining, Qinghai 810007, P.R.~China}\\
{\small $^{b}$Department of Mathematics, West Virginia University,}\\
{\small  Morgantown, WV, USA}}

\date{}

\maketitle

\noindent {\bf Abstract:}\ \  For a graph $G$ with $n$ vertices, let $\nu(G)$
and $A(G)$ denote the matching number and  adjacency matrix of $G$, respectively.
The permanental   polynomial of $G$ is defined as $\pi(G,x)={\rm per}(Ix-A(G))$.
The permanental nullity of $G$, denoted by $\eta_{per}(G)$, is the multiplicity
of the zero root of $\pi(G,x)$. In this paper, we use the Gallai-Edmonds structure
theorem to derive a concise formula which reveals the relationship  between the
permanental nullity and the matching number of a graph. Furthermore, we prove a
necessary and
sufficient condition for a graph $G$ to have $\eta_{per}(G)=0$.
As  applications, we show that every  unicyclic graph $G$ on $n$ vertices satisfies
$n-2\nu(G)-1 \le \eta_{per}(G) \le n-2\nu(G)$,
that the permanental nullity of the line graph of a graph
is either zero or one, and that the permanental nullity of a factor critical  graph
is always zero.
\section{Introduction}

Let $G=(V(G), E(G))$ be a simple graph with $n$  vertices and $m$   edges.
The {\em neighborhood} of vertex $v\in V(G)$ in a graph $G$, denoted by $N_G(v)$
(or just $N(v)$, when $G$ is understood from the context),
is the set of vertices adjacent to $v$. For $T\subseteq V(G)$,
we use $G[T]$ to denote the {\em induced subgraph} of $G$ by $T$.
The {\em line graph} $L(G)$ of $G$ is the graph whose vertex set is $E(G)$, where
two
vertices of $L(G)$ are adjacent in $L(G)$ if and only if
the corresponding edges are adjacent in $G$.

A set $M$ of edges in $G$ is {\em a matching} if every vertex of
$G$ is incident with at most one edge in $M$. For two matchings $M$ and $N$,
the {\em symmetric difference} of $M$ and
$N$ is defined to be
$M\triangle N = (M\cup N)-(M\cap N)$.
A vertex $v$ said to be  {\em covered} (or {\em saturated}) by $M$ if some edge of $M$
is incident with $v$. A {\em maximum matching} is one which covers as many vertices as
possible. In particular, a  maximum matching covering all vertices of $G$ is called a
{\em perfect matching}. A {\em near-perfect} matching in a graph $G$ is one covering
all but exactly one vertex of $G$.
The size of a maximum matching in $G$ is called the {\em matching number} of $G$
and is denoted by $\nu(G)$. A graph $G$ is said to be {\em factor-critical} if
$G-v$ has a perfect matching for every $v\in V(G)$.

\indent
The {\em permanent} of an $n\times n$ matrix $A=(a_{ij})(i,j=1,2,\ldots,n)$ is
defined as $${\rm per}(A)=\sum_{\sigma}\prod_{i=1}^{n}a_{i\sigma(i)},$$
where the sum is taken over all permutations $\sigma$ of $\{1,2,\ldots, n\}$.
Valiant \cite{val}  showed that computing the permanent of a matrix is \#P-complete
even when restricted to (0, 1)-matrices.

For an $n$ by $n$ matrix $A$, define ${\rm per}(xI -A)$ to
be the permanental polynomial of $A$.
If $G$ is a graph and $A(G)$ is the {\em adjacency matrix} of $G$, then
{\em permanental polynomial} of $G$ is defined to be
$\pi(G,x)={\rm per}(xI -A(G))$. That is, the
permanental polynomial of $A(G)$.
The {\em permanental spectrum}  of   $G$, denoted by $ps(G)$,
is the collection of all roots (together with their multiplicities)
of $\pi(G,x)$.
The multiplicity of the zero root of $\pi(G,x)$,
denoted by $\eta_{per}(G)$, is called the
{\em permanental nullity} (per-nullity for short) of $G$.

It seems that the permanental polynomials of graphs were
first considered by Turner \cite{tur}.
Subsequently, Merris et al. \cite{mer1} and Kasum et al. \cite{kas}
systematically introduced permanental polynomial   and its potential
applications in mathematical and chemical studies,  respectively.
Since then, very few research papers on the permanental polynomial were
published for a period of time (see \cite{cas1}).
This may be due to the difficulty of
computing the permanent ${\rm per}(xI-A(G))$. However,
permanental polynomials and their applications have received
a lot of attention from researchers in recent years, as shown in
\cite{bel,bor2,cas2,chen,gee, gut0,gut2,huo,wu1,yan,zhp} and the references therein.

The {\em spectrum} of a graph (i.e., the roots of the characteristic polynomial
of a graph with their multiplicities. See \cite{cve}) encodes useful combinatorial
information of the  graph. The relationship between the structural properties of
a graph and its spectrum
has been studied extensively over the years. Nevertheless, only a few results
on the permanental spectrum have been published. Brenner and Brualdi \cite{bre}
proved the following: If $A$ is an $n$ by $n$ matrix with
nonnegative entries and spectral radius $\rho$, then every root of the permanental
polynomial of $A$ must be in $\{z:|z| \leq \rho\}$.
Merris \cite{mer} observed that if $A$ is hermitian
with eigenvalues $\lambda_{1} \geq \lambda_{2}\geq \ldots\geq \lambda_{n}$,
then each real permanental root of $A$ is
in the interval $[\lambda_{n},\lambda_{1}]$. Borowiecki \cite{boro}
proved that $G$ has $ps(G)= \{i\lambda_{1}, \ldots, i\lambda_{n}\}$ if and only
if $G$ is bipartite without cycles of length $4k$ ($k = 1, 2, \ldots$),
where $i$ is imaginary unit and
$\{\lambda_{1}, \ldots, \lambda_{n}\}$ is the adjacency spectrum  of $G$.
Zhang et al. \cite{zhp1} proved that every
graph does not have a negative real permanental root.
In particular, they showed that a bipartite graph has no real permanental
roots except possibly zero. Additionally, several papers have been published
on graphs uniquely determined by their permanental spectra, see \cite{liu,liu1,wu2,zhp2},
among others.

In \cite{wu}, Wu and Zhang introduced  the per-nullity of a graph,  and
presented some elementary properties of per-nullity.
Furthermore, they   characterized the extremal graphs of order $n$
whose per-nullities are $n-2$, $n-3$,  $n-4$ and  $n-5$, respectively.
It is natural to consider the problem of computing the per-nullity of
graphs. In this paper, we investigate the problem of computing
the per-nullity of graphs, and find a relationship
between per-nullity and matching number of a graph.
The rest of this paper is organized as follows. In Section 2,
we demonstrate some preliminaries on
per-nullity of graphs.  In Section 3, using the Gallai-Edmonds structure theorem,
we obtain a relationship between the  per-nullity and the matching number of a graph.
In Section 4, we determine all graphs with zero per-nullity. In the
last section, we apply our main results to several classes of graphs, including
unicyclic graphs, line graphs and factor critical graphs.

\section{Preliminaries}
A  {\em Sachs graph} is a simple graph such that
each component  is regular and has degree 1 or 2. In
other words, the components are single edges and cycles.
\begin{lemma}\label{art20}{(R. Merris et al. \cite{mer1})}
Let G be a graph with $\pi(G, x)=\sum\limits_{k=0}^{n}b_{k}(G)x^{n-k}$.  Then
$$b_{k}(G)=(-1)^{k}\sum_{H}2^{c(H)}, 1\leq k\leq n,$$
where the sum is taken over all Sachs subgraphs $H$ of $G$ on $k$ vertices, and $c(H)$ is the number of cycles in $H$.
\end{lemma}
Let $S(G)$ be a maximum Sachs subgraph  of $G$
(i.e., $S(G)$ has  the maximum number of vertices among all Sachs subgraph  of $G$).
By the definition of a Sachs graph, $S(G)$ has
three possible structures: a maximum  matching,  union of disjoint cycles,
 or  union of some  disjoint single edges and   cycles.
In \cite{wu}, two elementary properties of per-nullity of graphs are introduced as follows.
\begin{lemma}\label{art22}(T. Wu and H. Zhang \cite{wu})
Let $G$ be a simple graph with $n$ vertices.
\\
(i) $\eta_{per}(G)=n$ if and only if $G$ is an empty  graph.
\\
(ii) If  $G_{1}, G_{2},..., G_{t}$ are the connected components
of $G$, then $\eta_{per}(G)=\sum_{i=1}^{t}\eta_{per}(G_{i})$.
\end{lemma}

\begin{lemma}\label{art23}(T. Wu and H. Zhang \cite{wu})
Let $G$ be a graph with $n$ vertices and $S(G)$ be a maximum Sachs subgraph  of $G$. Then
$\eta_{per}(G)= n-|V(S(G))|$.
\end{lemma}

\indent In the following we present the famous Gallai-Edmonds structure theorem
on
matchings of graphs. Definition \ref{BCD}(i) comes from \cite{lov,yu}. The notation of
Definition \ref{BCD} (i) and (ii) will be used throughout this paper.

\begin{Definition} \label{BCD}
Let $G$ be a graph.
\\
(i) Let $D(G)$ be the set of all vertices in $G$
which are not saturated by at least one maximum matching of $G$.
Define $B(G) = \{v \in (V(G)-D(G)) \; :$ there exist a $u \in B(G)$ with $uv \in E(G)\}$.
Finally let $C(G)=V(G)-(D(G) \cup B(G))$.
This yields a vertex-partition of $V(G)$ into
$B(G)$, $C(G)$ and $D(G)$, which is well-defined for every graph and
does not depend on the choices of any maximum matching.
\\
(ii) Let $D_0'(G)$ be the set of all isolated vertices in
$G[D(G)]$ and
$\F(G)$ be the set of components in $G[D(G)]$ each of which has order at least 3.
\end{Definition}
With this partition, the Gallai-Edmonds structure theorem is stated as follows.

\begin{theorem}\label{art25}{(Gallai-Edmonds Structure Theorem \cite{lov,yu})}
Let $G$ be a
graph and let $B(G)$, $C(G)$ and $D(G)$ be the vertex-partition defined above.
Each of the following holds.
\\
(i) The components of the subgraph induced by $D(G)$ are factor-critical.
\\
(ii) The subgraph induced by $C(G)$ has a perfect matching.
\\
(iii) Any maximum matching $M$ of $G$ contains a near-perfect matching
of each component of $G[D(G)]$ and a perfect matching of each component of $G[C(G)]$, and
$M$ matches all vertices of $B(G)$ with vertices in distinct components of $G[D(G)]$.
\\
(iv) The size of maximum matching is $\frac{1}{2}(|V(G)|-c(D(G))+|B(G)|)$,
where $c(D(G))$ denotes the number of components of the graph spanned by $D(G)$.
\end{theorem}

By Theorem \ref{art25}, we obtain the following lemma, which will be used later
in our arguments.

\begin{lemma}\label{art255}
Let $G$ be a graph with $\F(G) \neq \emptyset$ and
without a perfect matching.
If a maximum matching $M$ of $G$ covers as many isolated vertices
in  $G[D(G)]$ as possible, then there must exist at least one
component  of $G[D(G)]$ in $\F(G)$ not covered  by $M$.
\end{lemma}

\begin{proof}
Since $G$ does not have a perfect matching,
it follows from Theorem \ref{art25}(iv) that $c(G[D(G)]) > |B(G)|$.
By Theorem \ref{art25}(iii),  $M$ contains a subset $M_{BD}$ which
matches $B(G)$ with vertices in distinct components of $D(G)$.
Let $W \subseteq D(G)$ be the set of vertices covered by $M_{BD}$.
Then $W$ consists of some isolate
vertices in $G[D(G)]$ and some vertices of components each of which
has at least 3 vertices in $G[D(G)]$. If
all vertices in $W$ are taken isolate vertices in $G[D(G)]$,
or if $G[D(G)]$ does not have any isolated vertices,
then the  conclusion follows from the fact that $c(G[D(G)]) > |B(G)|$
and the assumption that $\F(G) \neq \emptyset$.
Let $D_0''(G)$ be a subset of $D_0'(G)$. Define
$B'(G)=\{u\in B(G) ~:$  for some $w \in D(G)$, $uw\in M_{BD}\}$,
$B_1'(G) = \{v \in B'(G)~:$ for some $w \in D_0''(G)\subseteq D_0'(G)$, $vw\in M_{BD}\}$,
and $B_2'(G) = B'(G) - B_1'(G)$. Since the choice of $M$ maximizes $|B_1'(G)|$,
for every vertex $v \in B_2'(G)$, $N_G(v)  \cap (D_{0}'(G) -D_{0}''(G)) = \emptyset$.
It follows from
the definition of $D(G)$ that there must be a vertex in $B'(G)$
adjacent to at least two components in  $G[D(G)]$ that are in $\F(G)$.
The conclusion of the lemma now follows.
\end{proof}

\begin{theorem}\label{art26}{(G. Chartrand et al. \cite{cha})}
Let $G$ be a nontrivial connected graph. Then
the line graph $L(G)$ contains a
perfect matching if and only if $|E(G)| \equiv 0$ (mod 2).
\end{theorem}

\begin{cor}\label{art27}{(G. Chartrand et al. \cite{cha})}
The line graph $L(G)$ of a nontrivial graph $G$ has a perfect matching if and
only if every component of $L(G)$ has even order.
\end{cor}

Let $G$ be a connected graph with at least 3 edges and $|E(G)| \equiv 1$ (mod 2).
Then as $G$ has a spanning tree, $G$ has an edge $e$ such that $G - e$ is either
connected or has two components with one being an isolated vertex.
If $G$ is 2-edge-connected, then
for any $e \in E(G)$, $G-e$ is connected and has an even number of edges.
With these observations, we have the following consequence of
Theorem \ref{art26}.

\begin{theorem}\label{art28}
Let $G$ be a connected graph with at least 3 edges and with $|E(G)| \equiv 1$ (mod 2).
Each of the following holds.
\\
(i) The line graph $L(G)$ contains a near-perfect matching.
\\
(ii) If, in addition,  $G$ is 2-edge-connected, then $L(G)$ is factor-critical.
\end{theorem}

\section{A relationship between the  per-nullity and the matching number of graphs}

By  Lemma \ref{art22} and by working componentwise,
it suffices to discuss connected graphs in this sections.
We start with a lemma.

\begin{lemma}\label{art30}
Let $G$ be a factor-critical graph with $n \geq 3$ vertices.
Each of the following holds.
\\
(i) Every vertex $v\in V(G)$ lies in an odd cycle of $G$.
\\
(ii) There exist an odd cycle $C$ and a maximum
matching $M$ of $G$
such that $G$ is covered by $E(C) \cup M$ and such that
the maximum degree of $G[E(C) \cup (M - E(G[V(C)]))]$ is 2. (Thus
$G[E(C) \cup M]$ is a maximum Sachs subgraph of $G$.)
\end{lemma}

\begin{proof}
Since $G$ is a factor-critical graph, $|V(G)|$ is odd, and $G$ is connected
and not a bipartite graph. If $G$ is an odd cycle,
then  Lemma \ref{art30} is obvious. Thus, suppose that
$G$ is not a cycle below.

Let $uv$ be an edge of $G$. Since $G$ is factor-critical, $G-v$ has
a perfect matching $M_{v}$. Similarly, $G-u$  has a perfect matching $M_{u}$.
It follows that the symmetric difference
$M_{u}\triangle M_{v}$ contains exactly
one path $P$ of even length joining $u$ and $v$. By the
choices of $M_u$ and $M_v$, the edge $uv$ is not
in $P$, and so $C = P+uv$ is an odd containing $v$.
Since $M_{v}$ covers all vertices of $G - v$,
$E(C) \cup (M_v - E(G[V(C)]))$ is a cover of $G$. By the definition
of $M_v$, no edge in $M_v - E(G[V(C)])$ is incident with
an edge in $C$, and so the maximum degree of $G[E(C) \cup (M - E(G[V(C)]))]$ is 2.
This completes the proof of the lemma.
\end{proof}

\begin{lemma}\label{art31}
Let $G$ be a connected graph with  $n$ vertices
and with the size of a maximum matching being $\nu(G)$. The
following are equivalent.
\\
(i)  $\eta_{per}(G)= n-2\nu(G)$.
\\
(ii) Either $G$ has a perfect matching or $E(G[D(G)]) = \emptyset$.
\end{lemma}

\begin{proof}
Assume (i) to prove (ii). By Lemma \ref{art23}, the equality $\eta_{per}(G)= n-2\nu(G)$
implies that a maximum matching of $G$ is a maximum Sachs subgraph.
By the definition of $D(G)$, we observe that if $D(G) = \emptyset$, then $G$
has a perfect matching, and so (ii) holds. Hence we assume that
$D(G) \neq \emptyset$. Suppose that there exists  at least one
components of $G[D(G)]$ having at least 3 vertices.
By Lemma \ref{art255}, there must be at least one
component of $G[D(G)]$ in $\F(G)$ not covered  by $M$.
Let $\F_M(G)$ denote all such components.
By Lemma \ref{art30}, for each $L \in \F_M(G)$, there exists an odd cycle $C_L$ and
a subset $M_L \subset M \cap E(L)$, such that $E(C_L) \cup M_L$
covers $V(L)$ and such that the maximum degree of $L[E(C_L) \cup M_L]$ is at most 2.
Thus, $H = G[\cup_{\small L \in \F_M(G)} (E(C_L) \cup M_L)]$ is
a Sachs graph $H$ such that  $|V(H)|$ is more than  the number of vertices in $G$
covered by the maximum matching $M$.  This implies that any maximum matching of $G$
is not a maximum Sachs subgraph, contrary to the fact that any
maximum matching must also be a maximum Sachs subgraph.
Therefore, $\F(G) = \emptyset$  and so $E(G[D(G)]) = \emptyset$.

We now assume (ii) to prove (i). If $G$ has a perfect matching, then the perfect matching
is a maximum Sachs subgraph of $G$. By Lemma \ref{art23},
$\eta_{per}(G)= n-2\nu(G)$. Suppose that $G$ does not have a perfect matching
and $E(G[D(G)]) = \emptyset$.
Since every maximum matching of $G$ is a maximum Sachs subgraph of $G$,
it follows by Lemma \ref{art23} that $\eta_{per}(G)= n-2\nu(G)$.
\end{proof}

\begin{Definition} \label{FM}
For a maximum matching $M$ of $G$ such that
\begin{equation} \label{max-iso}
\mbox{ the number of
isolated vertices in $G[D(G)]$ covered by $M$ is maximized,}
\end{equation}
define $M(G)$ to be the number of components of order at
least 3 in $G[D(G)]$ each of which  has just a vertex not covered  by $M$.
\end{Definition}
By Theorem \ref{art25}(i), every graph in $\F(G)$ is factor-critical.
By Lemma \ref{art31}, if $\F(G) \neq \emptyset$, then $\eta_{per}(G)< n-2\nu(G)$.
The next lemma describes the per-nullity of the graphs with $\F(G) \neq \emptyset$.

\begin{lemma}\label{art32}
Let $G$ be a connected graph with $n$ vertices and without a perfect matching,
If  $\F(G) \neq \emptyset$, then
\[
\eta_{per}(G)=n-2\nu(G)-M(G).
\]
\end{lemma}

\begin{proof}
Let $M$ be a maximum matching of $G$ satisfying (\ref{max-iso}).
Since $G$ does not have a perfect matching,
by Theorem \ref{art25}, $c(D(G))>|B(G)|$.
Then there exists at least one $H \in \F(G)$ such that $H$ has just
a vertex not covered by $M$.
By Lemma \ref{art30}
every $H \in \F(G)$ has an odd cycle $C_H$ such that
$E(E(C_H)) \cup  (M - E(G[V(C_H)]))]$ is a cover of $H$.
It follows that $G$ has a maximum Sachs subgraph $S(G)$ consisting of disjoint
odd cycles $\{C_H: H \in \F(G)\}$, and a subset of $M$.
It is routine to verify that
$|S(G)| = 2\nu(G)+M(G)$, and so
by Lemma \ref{art23}, we have $\eta_{per}(G)=n-2\nu(G)-M(G)$.
\end{proof}

By applying Lemmas \ref{art31} and \ref{art32}, we obtain
the main result of this section. Recall that $D(G)$ and $M(G)$ are defined
in Definitions \ref{BCD} and \ref{FM}, for a given maximum matching $M$ of $G$.

\begin{theorem}\label{art33}
Let $G$ be a connected graph with $n$ vertices,
and let $M$ be a maximum matching of $G$ satisfying (\ref{max-iso}).
Then
\[
\eta_{per}(G)= \left\{
\begin{array}{ll}
n-2\nu(G) & \mbox{ if $G$ has a perfect matching or $\F(G) = \emptyset$,}
\\
n-2\nu(G)-M(G) & \mbox{ otherwise. }
\end{array} \right.
\]
\end{theorem}

%

\section{The graphs with zero per-nullity}

For a simple graph $G$ on $n$ vertices, it is known that
$0\leq \eta_{per}(G)\leq n-2$. In this section, we will characterize the
graphs  with zero per-nullity. Note that by Lemma \ref{art23},
\begin{equation} \label{0-per}
\mbox{
$\eta_{per}(G) = 0$ if and only if $G$ has a spanning Sachs subgraph. }
\end{equation}

\begin{theorem}\label{art35}
Let $n \ge 2$ be an integer, $G$ be a connected graph on $n$ vertices.
Then $\eta_{per}(G)=0$ if and only if one of the following holds:
\\
(i) $G$ has a perfect matching, or
\\
(ii)  $G[D(G)]$ has no isolated vertices, or
\\
(iii) $G[D(G)]$ has  isolated vertices and
$G$ has a maximum matching covering every isolated vertices of $G[D(G)]$.
\end{theorem}

\begin{proof}
Assume first that $G$ satisfies one of (i), (ii) and (iii). We are to show that
$\eta_{per}(G)=0$.

If (i) holds, then by Theorem \ref{art33}, $\eta_{per}(G)=0$.
Hence we may assume that $G$ has no perfect matchings, and so $|D(G)| > 0$.

Suppose (ii) holds.
Then $|\F(G)| = c(D(G))$.
By Lemma \ref{art30}, each $H \in \F(G)$ has
a Sachs subgraph $S(H)$ to cover $V(H)$. Let $M$ be a perfect matching of
$G[C(G)]$. Then $S(G)= M \cup (\cup_{H \in \F(G)} E(S(H)))$ is
a spanning Sachs subgraph of $G$, and so by (\ref{0-per}),
$\eta_{per}(G)= n - |V(S(G))| = 0$.

Finally, we assume that $G$ has a maximum matching $M$
covering every isolated vertices of $G[D(G)]$.
By Lemma \ref{art30}, a Sachs subgraph exists to cover every graph
in $\F(G)$. Edges in these Sachs subgraphs in the graphs of $\F(G)$ together
with a subset of $M$ induces a spanning Sachs subgraph of $G$.
By (\ref{0-per}),
$\eta_{per}(G)= n - |V(S(G))| = 0$.

Conversely, we assume that  $\eta_{per}(G)=0$ to
show (i) or (ii) or (iii) must occur.
Choose a maximum matching $M$ satisfying (\ref{max-iso}).
By Theorem \ref{art33}, if $M(G) = 0$, then the assumption $\eta_{per}(G)=0$ leads to
$|V(G)|=2\nu(G)$. In this case, $G$ has a perfect matching, and so (i) follows.
Hence we assume that $M(G) > 0$.
By Lemma \ref{art32}, $|V(G)|=2\nu(G)+M(G)$. This implies that either all components of order 1
in $G[D(G)]$ is covered by $M$, whence (ii) holds;  or every component of $G[D(G)]$ is in $\F(G)$,
whence (iii) follows.
This completes the proof of the theorem.
\end{proof}

\section{Some  applications}

In this section, we determine the per-nullity
of some classes of graphs as applications of Theorems \ref{art33} and \ref{art35}.
An {\em unicyclic} graph is a connected graph with
equal number of vertices and edges. The theorem below determines
the per-nullity of unicyclic graphs.

\begin{theorem}\label{art36}
Let $G$ be an unicyclic graph with $n$ vertices and the
unique cycle in $G$ is denoted by $C_{\ell}$. Then
\begin{equation*}
\eta_{per}(G)=
\begin{cases}
n-2\nu(G)-1 & \text{if  $\ell$ is odd and $\nu(G)=\frac{\ell-1}{2}+\nu(G-C_{\ell})$,}\\
n-2\nu(G)& \text{otherwise}.
\end{cases}
\end{equation*}
\end{theorem}
\begin{proof}
Since $C_{\ell}$ is the only one cycle in $G$, it is routine
to see that only $C_{\ell} \in \F(G)$ is a factor-critical component.
By Theorem \ref{art33}, we have
$n-2\nu(G)-1 \le \eta_{per}(G) \le n-2\nu(G)$.

If $\eta_{per}(G)=n-2\nu(G)-1$, then by  Theorem \ref{art33},
there exists a maximum matching $M$ of $G$ satisfying (\ref{max-iso}),
and $|\F(G)| = 1$. Since $G$ is an unicyclic graph, the factor-critical
component of $G$ must be $C_{\ell}$. This implies that $C_{\ell}$ is odd.
By $(iii)$ of Theorem \ref{art25}, we have $\nu(G)=\frac{\ell-1}{2}+\nu(G-C_{\ell})$.

Assume that $C_{\ell}$ is odd and $\nu(G)=\frac{\ell-1}{2}+\nu(G-C_{\ell})$.
Then  $C_{\ell}$ is factor-critical, and there exists a maximum matching covering
all vertices of $C_{\ell}$ excepting a vertex. It follows from Theorem \ref{art33}
that $\eta_{per}(G)=n-2\nu(G)-1$.

By Theorem \ref{art33}, it is routine to verify
that in this case, $\eta_{per}(G)=n-2\nu(G)$ if and only if
$G$ has a perfect matching or if $\F(G) = \emptyset$,
The theorem now follows.
\end{proof}

\begin{theorem}\label{art37}
Let $G$ be an unicyclic graph with a unique cycle $C$.
Then $\eta_{per}(G)=0$ if and only if $G$ is an odd cycle,
$G$ has a perfect matching, or $G-V(C)$ has a perfect matching.
\end{theorem}
\begin{proof}
By (i) and (ii) of Theorem \ref{art35},
it is routine to verify that if $G$ is an odd cycle or $G$ has a perfect matching,
then $\eta_{per}(G)=0$. Thus we assume that
$G$ is not an odd cycle and $G$ does not have a perfect matching, and
that $G-V(C)$ has a perfect matching $M_C$.
Then $E(C) \cup M_C$ is a spanning Schas subgraph of $G$, and
so by (\ref{0-per}), $\eta_{per}(G)=0$.

Conversely, assume that  $\eta_{per}(G)=0$, and that
$G$ is not an odd cycle and $G$ does not have a perfect matching.
We are to show that $G-V(C)$ has a perfect matching.
By contradiction, suppose that $G - V(C)$ does not have a
perfect matching.  By Theorem \ref{art35}, (ii) or (iii)
of Theorem \ref{art35} must hold.
Since $G$ is an unicyclic graph, the cycle $C$ of $G$ must be the only
factor-critical  component order at least 3 in $G[D(G)]$. Hence
$|V(C)|$ is odd. If Theorem \ref{art35} (ii) holds, then $G$ must be an odd cycle,
contrary to the assumption that $G$ is not an odd cycle.
Hence Theorem \ref{art35} (iii) must hold.
By Theorem \ref{art25} (iii), $G-V(C)$ has a perfect matching.
This completes the proof of the theorem.
\end{proof}

By Theorems \ref{art26}, \ref{art28} and \ref{art33}, we obtain the following results.
\begin{theorem}\label{art34}
Let $L(G)$ be the line graph of $G$. Then
the per-nullity of $L(G)$ equals zero  or one.
\end{theorem}

\begin{theorem}
Let $G$ be a  factor-critical graph. Then $\eta_{per}(G)=0$.
\end{theorem}

\end{document}